\documentclass[12pt]{article}

\usepackage{amsmath, epsfig, cite}
\usepackage{amssymb}
\usepackage{amsfonts}
\usepackage{latexsym}
\usepackage{float}
\usepackage{color}
\usepackage{booktabs}
\usepackage{graphicx}
\usepackage{caption}
\usepackage{mathtools}
\usepackage[colorlinks=true,linkcolor=blue,citecolor=blue,urlcolor=blue]{hyperref}

\newtheorem{thm}{Theorem}

\newtheorem{cor}{Corollary}

\newtheorem{lem}{Lemma}[section]

\newcommand{\pf}{\noindent{\it Proof.} }

\numberwithin{equation}{section}
\allowdisplaybreaks

\newcommand{\qed}{\hfill$\square$\medskip}

\begin{document}

	\begin{center}
	{\large\bf  Two $q$-Supercongruences Related to Double Sums  }
\end{center}
\vskip 2mm \centerline{Wei-Wei Qi}

\begin{center}
	{\footnotesize MOE-LCSM, School of Mathematics and Statistics, Hunan Normal University, Hunan 410081, P.R. China\\[5pt]
		{\tt wwqi2022@foxmail.com} \\[10pt]
	}
\end{center}

\vskip 0.7cm \noindent{\bf Abstract.} 
In this paper, by using $q$-identities and  differential operator techniques, we establish two $q$-supercongruences modulo the square of a cyclotomic polynomial, which are associated with truncated double basic hypergeometric $q$-series.

\vskip 2mm
\noindent{\it MR Subject Classifications}: 33D15, 11A07, 11B65	

\section{Introduction} 

For any complex number $x$ and nonnegative integer $n$, define the shifted-factorial as
\[
\left(x\right)_n=\Gamma(x+n)/\Gamma(x),
\]
where $\Gamma(x)$ is the Gamma function.  In $1914$, Ramanujan \cite{S} (see also \cite[p. 352]{B-2}) discovered the identity
\begin{equation}\label{In-1}
\sum_{k=0}^{\infty}(6k+1)\frac{\left(\frac{1}{2}\right)_k^3}{k!^34^k}=\frac{4}{\pi},
\end{equation}	 
which was first proved by J. M. Borwein and P. B. Borwein \cite[pp. 177--187]{B-3}. In 1997, Van Hamme \cite{B-4} conjectured an interesting $p$-adic analogue of \eqref{In-1}: for any prime $p>3$,
\begin{equation}\label{In-1-0}
\sum_{k=0}^{(p-1)/2}(6k+1)\frac{\left(\frac{1}{2}\right)_k^3}{k!^34^k}\equiv (-1)^{(p-1)/2}p \pmod{p^4},
\end{equation}
The above congruence was later confirmed by Long \cite{B-5}. Meanwhile, in the same paper, Long also posed a congruence involving double sums: for any odd prime $p$,
\begin{equation}\label{In-2}
	\sum_{k=0}^{(p-1)/2}(-1)^k(6k+1)\frac{\left(\frac{1}{2}\right)_k^3}{k!^38^k}
	\sum_{j=1}^{k}\left(\frac{1}{(2j-1)^2}-\frac{1}{16j^2}\right)\equiv 0 \pmod{p},
\end{equation}	 
which was subsequently verified by Swisher \cite{B-6}.

For any complex numbers $a$, nonnegative integer $n$, the $q$-shifted factorial be defined as
\[
(a;q)_{\infty}=\prod_{j=0}^{\infty}(1-aq^j) \quad and \quad (a;q)_n=\frac{(a;q)_{\infty}}{(aq^n;q)_{\infty}}.
\]
For succinctness, we use the shorthand notation
\[
	(a_1,a_2,a_3,\cdots, a_t;q)_n=(a_1;q)_n(a_2;q)_n(a_3;q)_n\cdots(a_t;q)_n,
\]
where $t\in \mathbb{Z}^+$, and $n\in \mathbb{Z}^+\cup \{0,\infty\}$. 
For $n\in \mathbb{Z}^+$, the $q$-integer is defined as
\[
[n]=[n]_q=\frac{1-q^n}{1-q}=1+q+q^2+\cdots+q^{n-1}.
\]	
In addition, for a primitive $n$-th root of unity $\zeta$,  the $n$-th cyclotomic polynomial is given by
\begin{align*}
	\Phi_n(q)=\prod_{\substack{1\le k \le n\\[3pt](n,k)=1}}(q-\zeta^k),
\end{align*}

In $2020$, Gu and Guo \cite{B-7} gave a $q$-analogue of \eqref{In-2}: for any positive odd integer $n$,
\begin{equation}\label{In-3}
	\sum_{k=0}^{(n-1)/2}(-1)^k[6k+1]\frac{(q;q^2)_k^3}{(q^4;q^4)_k^3}
	\sum_{j=1}^{k}\left(\frac{q^{2j-1}}{[2j-1]^2}-\frac{q^{4j}}{[4j]^2}\right)\equiv 0 \pmod{\Phi_n(q)}.
\end{equation}
 Tang \cite[Theorem 1.1]{B-8} established two $q$-congruences on truncated double basic hypergeometric  series similar to \eqref{In-3}. For example, let $n$ be an odd integer. Then, modulo $\Phi_n(q)^2$, 
\begin{align}\label{In-4}
	\sum_{k=0}^{n-1}&[6k+1]\frac{(q;q^2)_k^2(q^{2};q^4)_k(q;q^4)_k^2}{(q^2;q^2)_k^2(q^{4};q^2)_k^3}q^{2k}\left(\sum_{i=1}^{k}\frac{q^{2i-1}}{[2i-1]^2}-\sum_{i=1}^{k}\frac{q^{4i}}{[4i]^2}\right) \notag \\
	&\equiv
	\begin{cases}
		[n]\frac{(q^2;q^4)_{(n-1)/4}^2}{(q^4;q^4)_{(n-1)/4}^2}\sum_{i=1}^{(n-1)/2}\frac{(-1)^{i-1}q^{2i}}{[2i]^2},& \text{if }   \quad  n\equiv 1 \pmod{4} ,\\
		\frac{[3n](q^2;q^4)_{(3n-1)/4}^2}{4[n]^2(q^4;q^4)_{(3n-1)/4}^2},& \text{if }  \quad  n\equiv 3 \pmod{4} .
	\end{cases}
\end{align}	 
Moreover, Tang \cite[Conjectures 4.1]{B-8} conjectured that \eqref{In-4}  hold true modulo $\Phi_n(q)^3$. Later, Luo, Wang and You \cite[Theorem 1.1]{B-9} extended \eqref{In-4} to the case modulo $\Phi_n(q)^3$.

Quite recently, Guo and Zhao derived the following congruence analogues to \eqref{In-1-0} by means of $q$-supercongruences: for any prime $p$, modulo $p^4$,
	\begin{align*}
	\sum_{k=0}^{K}&(6k+1)\frac{\left(\frac{1}{2}\right)_k^3\left(-\frac{1}{4}\right)_k}
	{k!^34^k(k+1)!} \notag\\
	&\equiv
	\begin{cases}
		-3p^3,& \text{if }   \quad    K=n-1 \quad and  \quad n\equiv 1\pmod{4} ,\\
		-p\frac{\left(-\frac{1}{2}\right)_{(p+1)/4}}{\left(1\right)_{(p+1)/4}}\left(1-\frac{p^2}{16}\sum_{j=1}^{(p+1)/4}\frac{1}{j^2}\right),& \text{if }    \quad  K=(n-1)/2 \quad  and \quad n\equiv 3\pmod{4}.
	\end{cases}
\end{align*}	 

In recent years, congruences involving double series have been an interesting topic
that attracted the attention of many researchers. For more $q$-supercongruences concerning double sums, the readers are referred to [\citen{W-2}, \citen{W-3}, \citen{W-4}, \citen{W-5}] and other relevant literature.
Inspired by the aforementioned work, the main objective of this paper is to establish new $q$-supercongruences involving double basic hypergeometric sums, which are analogous to \eqref{In-3}.

\begin{thm}\label{thm-1}
	Let $n$ be an odd positive integer with $n\equiv 1 \pmod{4}$. Then modulo $\Phi_n(q)^2$,
	\begin{align}\label{th-1}
		\sum_{k=0}^{M}&[6k+1]\frac{(q;q^2)_k^3(q^{-1};q^4)_k}{(q^4;q^4)_k^3(q^{4};q^2)_k}q^{k^2+3k}\left(\sum_{i=1}^{k}\frac{q^{2i-1}}{[2i-1]^2}-\sum_{i=1}^{k}\frac{q^{4i}}{[4i]^2}\right) \notag \\
		&\equiv
		\begin{cases}
			0,& \text{if }   \quad    M=(3n+1)/4 \quad and  \quad n>5 ,\\
			-\frac{3q(1+q)[n]}{2},& \text{if }    \quad  M=n-1 \quad  and \quad n\geq 5.
		\end{cases}
	\end{align}	 
\end{thm}

Letting $n=p$ be a prime and taking the limit as $q\rightarrow 1$ in Theorem \ref{thm-1}, we obtain the following congruence.
\begin{cor}
	Let $p$ be a  prime with $p\equiv 1 \pmod{4}$. Then modulo $p^2$,
	\begin{align*}
	\sum_{k=0}^{N}(6k+1)&\frac{\left(\frac{1}{2}\right)_k^3\left(-\frac{1}{4}\right)_k}{k!^3\left(k+1\right)!4^k}
	\left(\sum_{i=1}^{k}
	\frac{1}{(2i-1)^2}-\sum_{i=1}^{k}\frac{1}{16i^2}\right) \notag \\
	&\equiv
	\begin{cases}
		0,& \text{if }   \quad    N=(3p+1)/4 \quad and  \quad p>5 ,\\
		-3p,& \text{if }    \quad  N=p-1 \quad  and \quad p\geq 5.
	\end{cases}
\end{align*}	
\end{cor}

\begin{thm}\label{thm-2}
	For any odd positive integer $n$ with $n\equiv 3 \pmod{4}$, we have
	\begin{align} \label{th-2}
			\sum_{k=0}^{(n-1)/2}&[6k+1]\frac{(q;q^2)_k^3(q^{-1};q^4)_k}{(q^4;q^4)_k^3(q^{4};q^2)_k}q^{k^2+3k}\left(\sum_{j=1}^{k}\frac{q^{2j-1}}{[2j-1]^2}-\sum_{j=1}^{k}\frac{q^{4j}}{[4j]^2}\right) \notag \\
			&\equiv -\frac{(q^3,q^5;q^4)_{(n+1)/4}}{(q^4;q^4)_{(n+1)/4}^2}\sum_{j=1}^{(n+1)/4}\frac{q^{4j}}{[4j]^2}  \pmod{\Phi_n(q)^2}.
	\end{align}	
\end{thm}

Letting $n=p$ be a prime and taking the limit as $q\rightarrow 1$ in Theorem \ref{thm-2}, we obtain the following consequence.
\begin{cor}
	For any prime $p$ with $p\equiv 3 \pmod{4}$, we have
\[
			\sum_{k=0}^{(p-1)/2}(6k+1)\frac{\left(\frac{1}{2}\right)_k^3\left(-\frac{1}{4}\right)_k}{k!^3\left(k+1\right)!4^k}
			\sum_{i=1}^{k}\left(\frac{1}{(2i-1)^2}-\frac{1}{16i^2}
			\right)
			\equiv -\frac{\left(\frac{3}{2}\right)_{(p+1)/2}}{2^{(p+9)/2}\left(1\right)_{(p+1)/4}^2}
			\sum_{j=1}^{(p+1)/4}\frac{1}{j^2}  \pmod{p^2}. 
\]	
\end{cor}

The rest of the paper is organized as follows. In Section \ref{main-2}, we present  preliminary results. The proofs of Theorem \ref{thm-1} and Theorem \ref{thm-2} are given in Section \ref{main-3}.

\section{Preliminary results}\label{main-2}
In  2022, by using summation of Gasper and Rahman \cite[(3.8.12)]{W-1}, Wei \cite[(2.2)]{W} established an identity
\[
	\sum_{k=0}^{\infty}\frac{1-aq^{3k}}{1-a}
	\frac{(a,b,q/b;q)_k\,(f;q^2)_k\,q^{(k^2+k)/2}}
	{(q^2,aq^2/b,abq;q^2)_k\,(aq/f;q)_k}
	\left(\frac{a}{f}\right)^k 
	=\frac{(aq,aq^2,aq^2/bf,abq/f;q^2)_\infty}
	{(aq/f,aq^2/f,aq^2/b,abq;q^2)_\infty}.
\]
 Applying the above identity, Guo and Zhao \cite{G-Z} presented the following results.

\begin{lem}		
	Let $n>1$ be an integer with $n\equiv 1 \pmod{4}$.	Then
	\begin{equation}
			\sum_{k=0}^{(3n+1)/4}[6k+1]\frac{(q,q^{1+n},q^{1-n};q^2)_k(q^{-1}/b^3;q^4)_k}{(q^4,q^{4+n},q^{4-n};q^4)_k(b^3q^{4};q^2)_k}q^{k^2+3k}b^{3k}=\frac{(q^5,b^{-3};q^4)_{(n-1)/4}(b^{3}q)^{(n-1)/4}}{(b^3q^6,q;q^4)_{(n-1)/4}}.  \label{G-1-2}
	\end{equation}	
	\begin{equation}	
			\sum_{k=0}^{(3n+1)/4}[6k+1]\frac{(q,aq,q/a;q^2)_k(q^{-1-3n};q^4)_k}{(q^4,aq^4,q^4/a;q^4)_k(q^{4+3n};q^2)_k}q^{k^2+3k+3nk}=\frac{(q^5,q^3;q^4)_{(3n+1)/4}}{(aq^4,q^4/a;q^4)_{(3n+1)/4}}.  \label{G-1}
	\end{equation}	
	Let $n>1$ be an integer such that $n\equiv 3 \pmod{4}$, 
	\begin{equation}	
			\sum_{k=0}^{(n-1)/2}[6k+1]\frac{(q,aq,q/a;q^2)_k(q^{-1-n};q^4)_k}{(q^4,aq^4,q^4/a;q^4)_k(q^{4+n};q^2)_k}q^{k^2+3k+nk}=\frac{(q^5,q^3;q^4)_{(n+1)/4}}{(aq^4,q^4/a;q^4)_{(n+1)/4}}. \label{G-2}
	\end{equation}	
	\begin{equation}
			\sum_{k=0}^{(n-1)/2}[6k+1]\frac{(q,q^{1+n},q^{1-n};q^2)_k(q^{-1}/b;q^4)_k}{(q^4,q^{4+n},q^{4-n};q^4)_k(q^{4}b;q^2)_k}q^{k^2+3k}b^k=\frac{(q^3,1/bq^{2};q^4)_{(n+1)/4}(bq)^{(n+1)/4}}{(bq^4,q^{-1};q^4)_{(n+1)/4}}. \label{G-2-1}
	\end{equation}	
\end{lem}		
\pf \eqref{G-1-2} and \eqref{G-1} follow from \cite[(2.7) and (2.9)]{G-Z}, respectively. \eqref{G-2} and \eqref{G-2-1} appear in \cite[Page 10]{G-Z}.

\section{Proofs of the Theorems  }\label{main-3}
Given a multivariate function $f(x_1,x_2,\cdots, x_t)$, we define the partial  derivative operator $\mathcal{L}_{x_i}$ as
\[
	\mathcal{L}_{x_i}f(x_1,x_2,\cdots, x_t):=\frac{d}{dx_i}f(x_1,x_2,\cdots, x_t), \quad  \quad  1\leq i\leq t.
\]

We now present a parametric extension of Theorem \ref{thm-1} and Theorem \ref{thm-2}.
\begin{thm}
	(I). Let $n>5$ be a positive integer with $n\equiv 1 \pmod{4}$. Then modulo $\Phi_n(q)(b-q^n)$,
	\begin{equation}
			\sum_{k=0}^{(3n+1)/4}[6k+1]\frac{(q;q^2)_k^3(q^{-1}/b^3;q^4)_k}{(q^4;q^4)_k^3(q^{4}b^3;q^2)_k}q^{k^2+3k}b^{3k}
			\left(\sum_{i=1}^{k}
			\frac{q^{2i-1}}{[2i-1]^2}-\sum_{i=1}^{k}\frac{q^{4i}}{[4i]^2}\right)\equiv 0. \label{tl-1}
	\end{equation}	
	(II). Let $n$ be a positive integer such that $n\equiv 3 \pmod{4}$. Then modulo $\Phi_n(q)(b-q^n)$,
	\begin{align}
			\sum_{k=0}^{(n-1)/2}&[6k+1]\frac{(q;q^2)_k^3(q^{-1}/b;q^4)_k}{(q^4;q^4)_k^3(q^{4}b;q^2)_k}q^{k^2+3k}b^{k}\sum_{j=1}^{k}
			\left(\frac{q^{2j-1}}{[2j-1]^2}-\frac{q^{4j}}{[4j]^2}\right) \notag\\
			& \equiv -\frac{(q^3,q^5;q^4)_{(n+1)/4}}{(q^4;q^4)_{(n+1)/4}^2}\sum_{j=1}^{(n+1)/4}\frac{q^{4j}}{[4j]^2}. \label{tl-2}
	\end{align}
\end{thm}

\pf For $n\equiv 1 \pmod{4}$, by applying the operator $\mathcal{L}_{a}$	twice to \eqref{G-1}, we obtain
\begin{align}
		\sum_{k=0}^{(3n+1)/4}&[6k+1]\frac{(q,aq,q/a;q^2)_k(q^{-1-3n};q^4)_k}{(q^4,aq^4,q^4/a;q^4)_k(q^{4+3n};q^2)_k}q^{k^2+3k+3nk}\left(f_k(a)^2+\mathcal{L}_af_k(a)\right)  \notag\\
		&=\frac{(q^3;q^2)_{(3n+1)/2}}{(aq^4,q^4/a;q^4)_{(3n+1)/4}}\left(g_n(a)^2+\mathcal{L}_ag_n(a)\right), \label{0-1}
\end{align}			
where 
\begin{align*}
		&f_k(a)=\sum_{i=1}^{k}\left(\frac{q^{2i-1}}{a(a-q^{2i-1})}-\frac{q^{2i-1}}
		{1-aq^{2i-1}}-\frac{q^{4i}}{a(a-q^{4i})}+\frac{q^{4i}}{1-aq^{4i})}\right),\notag\\
		&g_n(a)=\sum_{i=1}^{(3n+1)/4}\left(\frac{q^{4i}}{1-aq^{4i}}-\frac{q^{4i}}{a(a-q^{4i})}\right).
\end{align*}	
Setting $a=1$ in \eqref{0-1} and simplifying, we deduce that
\begin{align*}	
		\sum_{k=0}^{(3n+1)/4}&[6k+1]\frac{(q;q^2)_k^3(q^{-1-3n};q^4)_k}{(q^4;q^4)_k^3(q^{4+3n};q^2)_k}q^{k^2+3k+3nk}
		\left(\sum_{i=1}^{k}
		\frac{q^{2i-1}}{[2i-1]^2}-\sum_{i=1}^{k}\frac{q^{4i}}{[4i]^2}\right) \notag\\
		&=-\frac{(q^3;q^2)_{(3n+1)/2}}{(q^4;q^4)_{(3n+1)/4}^2}\sum_{i=1}^{(3n+1)/4}\frac{q^{4i}}{[4i]^2}. 
\end{align*}	
This establishes the $q$-congruence: modulo $b-q^n$,
\begin{align}
		\sum_{k=0}^{(3n+1)/4}&[6k+1]\frac{(q;q^2)_k^3(q^{-1}/b^3;q^4)_k}{(q^4;q^4)_k^3(q^{4}b^3;q^2)_k}q^{k^2+3k}b^{3k}
		\left(\sum_{i=1}^{k}
		\frac{q^{2i-1}}{[2i-1]^2}-\sum_{i=1}^{k}\frac{q^{4i}}{[4i]^2}\right) \notag\\
		&\equiv -\frac{(q^3;q^2)_{(3n+1)/2}}{(q^4;q^4)_{(3n+1)/4}^2}\sum_{i=1}^{(3n+1)/4}\frac{q^{4i}}{[4i]^2}.  \label{0-2}
\end{align}	

In view of \cite[Theorem 2.3]{G-Z}: let $n>1$ be an integer with $n\equiv 1 \pmod{4}$. Then modulo $\Phi_n(q)(1-aq^n)(a-q^n)(b-q^n)$,
\begin{align*}
		\sum_{k=0}^{(3n+1)/4}&[6k+1]\frac{(q,aq,q/a;q^2)_k(q^{-1}/b^3;q^4)_k}{(q^4,aq^4,q^4/a;q^4)_k(q^{4}b^3;q^2)_k}q^{k^2+3k}b^{3k} \notag\\
		&\equiv\frac{(b-q^n)(ab-1-a^2+aq^n)}{(a-b)(1-ab)}\frac{(q^5,b^{-3};q^4)_{(n-1)/4}(b^3q)^{(n-1)/4}}{(b^3q^6,q;q^4)_{(n-1)/4}} \notag \\
		&+\frac{(1-aq^n)(a-q^n)}{(a-b)(1-ab)}\frac{(q^3,q^5;q^4)_{(3n+1)/4}}{(aq^4,q^4/a;q^4)_{(3n+1)/4}}. 
\end{align*}	
It is clear that 
\begin{equation}	
		\frac{(b-q^n)(ab-1-a^2+aq^n)}{(a-b)(1-ab)}\equiv 1 \pmod{(1-aq^n)(a-q^n)}. \label{sec-0-2}
\end{equation}	
It follows that for any positive integer $n\equiv 1 \pmod{4}$, we have the following result: modulo $\Phi_n(q)(1-aq^n)(a-q^n)$,
\[
		\sum_{k=0}^{(3n+1)/4}[6k+1]\frac{(q,aq,q/a;q^2)_k(q^{-1}/b^3;q^4)_k}{(q^4,aq^4,q^4/a;q^4)_k(q^{4}b^3;q^2)_k}q^{k^2+3k}b^{3k}
		\equiv\frac{(q^5,b^{-3};q^4)_{(n-1)/4}(b^3q)^{(n-1)/4}}{(b^3q^6,q;q^4)_{(n-1)/4}}. 
\]
Setting $a=1$ in the above $q$-congruence yields
\begin{align}	
		\sum_{k=0}^{(3n+1)/4}&[6k+1]\frac{(q;q^2)_k^3(q^{-1}/b^3;q^4)_k}{(q^4;q^4)_k^3(q^{4}b^3;q^2)_k}q^{k^2+3k}b^{3k} \notag\\
		&\equiv\frac{(q^5,b^{-3};q^4)_{(n-1)/4}(b^3q)^{(n-1)/4}}{(b^3q^6,q;q^4)_{(n-1)/4}} \pmod{\Phi_n(q)^3}.  \label{0-2-0}
\end{align}
Utilizing \eqref{G-1-2}, we obtain
\begin{align}
		\sum_{k=0}^{(3n+1)/4}&[6k+1]\frac{(q;q^2)_k^3(q^{-1}/b^3;q^4)_k}{(q^4;q^4)_k^3(q^{4}b^3;q^2)_k}q^{k^2+3k}b^{3k}-\frac{(q^5,b^{-3};q^4)_{(n-1)/4}(b^3q)^{(n-1)/4}}{(b^3q^6,q;q^4)_{(n-1)/4}} \notag \\
		&=\sum_{k=0}^{(3n+1)/4}[6k+1]\frac{(q;q^2)_k(q^{-1}/b^3;q^4)_k}{(q^4;q^4)_k(q^{4}b^3;q^2)_k}q^{k^2+3k}b^{3k}\mathcal{A}_k(q), \label{0-2-2}
\end{align}
where 
\begin{align*}
	\mathcal{A}_k(q)=\frac{(q;q^2)_k^2(q^{4+n},q^{4-n};q^4)_k-(q^4;q^4)_k^2(q^{1+n},q^{1-n};q^2)_k}{(q^4;q^4)_k^2(q^{4+n},q^{4-n};q^4)_k}	
\end{align*}

For any nonzero integer $r$, we have
\begin{equation}
	1-q^{rn}\equiv 0 \pmod{\Phi_n(q)}, \label{0-2-1}
\end{equation}	
and
\[
	(1-q^{mi+n})(1-q^{mi-n})=(1-q^{mi})^2-(1-q^n)^2q^{mi-n}.
\]
Therefore, we deduce that
\begin{align*}	
		(q^{4+n},q^{4-n};q^4)_k&=\prod_{i=1}^{k}(1-q^{4i+n})(1-q^{4i-n}) \notag\\
		&\equiv (q^4;q^4)_k^2-(q^4;q^4)_k^2\sum_{i=1}^{k}\frac{(1-q^n)^2}{(1-q^{4i})^2}q^{4i-n}\pmod{\Phi_n(q)^4}.
\end{align*}
Similarly, there holds the following consequence:
\[
		(q^{1+n},q^{1-n};q^2)_k\equiv (q;q^2)_k^2-(q;q^2)_k^2\sum_{i=1}^{k}\frac{(1-q^n)^2}{(1-q^{2i-1})^2}q^{2i-1-n}\pmod{\Phi_n(q)^4}.
\]
It thus follows that
\begin{equation}	
		\mathcal{A}_k(q)\equiv [n]^2\frac{(q;q^2)_k^2}{(q^4;q^4)_k^2} 
		\sum_{i=1}^{k} \left(\frac{q^{2i-1-n}}{[2i-1]^2}-\frac{q^{4i-n}}{[4i]^2}\right) \pmod{\Phi_n(q)^4}. \label{0-2-3}
\end{equation}
Substituting \eqref{0-2-0} and \eqref{0-2-3} into \eqref{0-2-2} and simplifying, we obtain
\begin{equation}
		\sum_{k=0}^{(3n+1)/4}[6k+1]\frac{(q;q^2)_k^3(q^{-1}/b^3;q^4)_kq^{k^2+3k}b^{3k}}{(q^4;q^4)_k^3(q^{4}b^3;q^2)_k}\sum_{i=1}^{k}\left(\frac{q^{2i-1}}{[2i-1]^2}-\frac{q^{4i}}{[4i]^2}\right)\equiv 0 \pmod{\Phi_n(q)}. \label{0-2-4}
\end{equation}
Note that  $\Phi_n(q)$ and $(b-q^n)$ are coprime.  With the help of \eqref{0-2} and \eqref{0-2-4}, we derive the desired $q$-congruence, modulo $\Phi_n(q)(b-q^n)$,
\begin{align}
		\sum_{k=0}^{(3n+1)/4}&[6k+1]\frac{(q;q^2)_k^3(q^{-1}/b^3;q^4)_k}{(q^4;q^4)_k^3(q^{4}b^3;q^2)_k}q^{k^2+3k}b^{3k}
		\left(\sum_{i=1}^{k}
		\frac{q^{2i-1}}{[2i-1]^2}-\sum_{i=1}^{k}\frac{q^{4i}}{[4i]^2}\right) \notag \\
		&\equiv -\frac{(q^3;q^2)_{(3n+1)/2}}{(q^4;q^4)_{(3n+1)/4}^2}\sum_{i=1}^{(3n+1)/4}\frac{q^{4i}}{[4i]^2}.  \label{0-2-5}
\end{align}	
Since $(q^3;q^2)_{(3n+1)/2}$ contains the factor $(1-q^n)(1-q^{3n})$, and $(q^4;q^4)_{(3n+1)/4}$ is coprime to $\Phi_n(q)$. Consequently, by means of \eqref{0-2-1}, we conclude that the right-hand side of \eqref{0-2-5} vanishes modulo $\Phi_n(q)^2$, which implies that \eqref{tl-1} holds true modulo $\Phi_n(q)(b-q^n)$.\\

Furthermore, when $n\equiv 3 \pmod{4}$,  applying the operator $\mathcal{L}_a$ twice to \eqref{G-2} yields
\begin{align}
		\sum_{k=0}^{(n-1)/2}&[6k+1]\frac{(q,aq,q/a;q^2)_k(q^{-1-n};q^4)_k}{(q^4,aq^4,q^4/a;q^4)_k(q^{4+n};q^2)_k}q^{k^2+3k+nk}\left(h_k(a)^2+\mathcal{L}_ah_k(a)\right)\notag \\
		&=\frac{(q^5,q^3;q^4)_{(n+1)/4}}{(aq^4,q^4/a;q^4)_{(n+1)/4}}\left(l_n(a)^2+\mathcal{L}_al_n(a)\right), \label{0-3}
\end{align}		
where
\begin{align*}
		&h_k(a)=\sum_{j=1}^{k}\left(\frac{q^{2j-1}}{a(a-q^{2j-1})}-\frac{q^{2j-1}}{1-aq^{2j-1}}+\frac{q^{4j}}{1-aq^{4j}}-\frac{q^{4j}}{a(a-q^{4j})}\right),\notag\\
		&l_n(k)=\sum_{j=1}^{(n+1)/4}\left(\frac{q^{4j}}{1-aq^{4j}}-\frac{q^{4j}}{a(a-q^{4j})}\right).	
\end{align*}		
Putting $a=1$ in \eqref{0-3}, we arrive at
\begin{align*}
		\sum_{k=0}^{(n-1)/2}&[6k+1]\frac{(q;q^2)_k^3(q^{-1-n};q^4)_k}{(q^4;q^4)_k^3(q^{4+n};q^2)_k}q^{k^2+3k+nk}
		\left(\sum_{k=1}^{k}\frac{q^{2j-1}}{[2j-1]^2}-\sum_{k=1}^{k}\frac{q^{4j}}{[4j]^2}\right)\notag \\
		&=-\frac{(q^3,q^5;q^4)_{(n+1)/4}}{(q^4;q^4)_{(n+1)/4}^2}\sum_{k=1}^{(n+1)/4}\frac{q^{4j}}{[4j]^2}. 
\end{align*}	
Hence, we have the following $q$-congruence, modulo $b-q^n$,
\begin{align}	
		\sum_{k=0}^{(n-1)/2}&[6k+1]\frac{(q;q^2)_k^3(q^{-1}/b;q^4)_k}{(q^4;q^4)_k^3(q^{4}b;q^2)_k}q^{k^2+3k}b^k
		\left(\sum_{k=1}^{k}\frac{q^{2j-1}}{[2j-1]^2}-\sum_{k=1}^{k}\frac{q^{4j}}{[4j]^2}\right)\notag \\
		&\equiv -\frac{(q^3,q^5;q^4)_{(n+1)/4}}{(q^4;q^4)_{(n+1)/4}^2}\sum_{j=1}^{(n+1)/4}\frac{q^{4j}}{[4j]^2}. \label{0-4}
\end{align}	

In view of \cite[Theorem 4.1]{G-Z}: let $n$ be a positive integer with $n\equiv 3 \pmod{4}$. Then modulo $\Phi_n(q)(1-aq^n)(a-q^n)(b-q^n)$,
\begin{align*}	
		\sum_{k=0}^{(n-1)/2}&[6k+1]\frac{(q,aq,q/a;q^2)_k(q^{-1}/b;q^4)_k}{(q^4,aq^4,q^4/a;q^4)_k(q^{4}b;q^2)_k}q^{k^2+3k}b^{k} \notag\\
		&\equiv\frac{(b-q^n)(ab-1-a^2+aq^n)}{(a-b)(1-ab)}\frac{(q^3,1/bq^{2};q^4)_{(n+1)/4}(bq)^{(n+1)/4}}{(bq^4,q^{-1};q^4)_{(n+1)/4}} \notag\\
		&+\frac{(1-aq^n)(a-q^n)}{(a-b)(1-ab)}\frac{(q^3,q^5;q^4)_{(n+1)/4}}{(aq^4,q^4/a;q^4)_{(n+1)/4}}. 
\end{align*}	
Combining the above $q$-congruence and \eqref{sec-0-2}, we deduce that for any positive integer $n\equiv 3 \pmod{4}$, modulo $\Phi_n(q)(1-aq^n)(a-q^n)$, 
\[
		\sum_{k=0}^{(n-1)/2}[6k+1]\frac{(q,aq,q/a;q^2)_k(q^{-1}/b;q^4)_k}{(q^4,aq^4,q^4/a;q^4)_k(q^{4}b;q^2)_k}q^{k^2+3k}b^{k}
		\equiv\frac{(q^3,1/bq^{2};q^4)_{(n+1)/4}(bq)^{(n+1)/4}}{(bq^4,q^{-1};q^4)_{(n+1)/4}}.
\]	
Taking $a=1$, we arrive at
\begin{equation}	
		\sum_{k=0}^{(n-1)/2}[6k+1]\frac{(q;q^2)_k^3(q^{-1}/b;q^4)_k}{(q^4;q^4)_k^3(q^{4}b;q^2)_k}q^{k^2+3k}b^{k}
		\equiv\frac{(q^3,1/bq^{2};q^4)_{(n+1)/4}(bq)^{(n+1)/4}}{(bq^4,q^{-1};q^4)_{(n+1)/4}} \pmod{\Phi_n(q)^3}. \label{0-5}
\end{equation}
Upon application of \eqref{G-2-1}, we obtain
\begin{align*}	
		\sum_{k=0}^{(n-1)/2}&[6k+1]\frac{(q;q^2)_k^3(q^{-1}/b;q^4)_k}{(q^4;q^4)_k^3(q^{4}b;q^2)_k}q^{k^2+3k}b^{k}
		-\frac{(q^3,1/bq^{2};q^4)_{(n+1)/4}(bq)^{(n+1)/4}}{(bq^4,q^{-1};q^4)_{(n+1)/4}}\notag\\ 
		&=\sum_{k=0}^{(n-1)/2}[6k+1]\frac{(q;q^2)_k(q^{-1}/b;q^4)_k}{(q^4;q^4)_k(q^{4}b;q^2)_k}q^{k^2+3k}b^{k}\mathcal{A}_k(q).
\end{align*}
This, together with \eqref{0-2-3} and \eqref{0-5}, gives
\begin{equation}
		\sum_{k=0}^{(n-1)/2}[6k+1]\frac{(q;q^2)_k^3(q^{-1}/b;q^4)_k}{(q^4;q^4)_k^3(q^{4}b;q^2)_k}q^{k^2+3k}b^{k}\sum_{j=1}^{k}\left(\frac{q^{2j-1}}{[2j-1]^2}-\frac{q^{4j}}{[4j]^2}\right) \equiv 0 \pmod{\Phi_n(q)}. \label{0-6}
\end{equation}
It is evident that $\Phi_n(q)$ and $b-q^n$ are coprime. Combining \eqref{0-4} and \eqref{0-6}, we arrive at \eqref{tl-2}.  \qed

\noindent\textbf{Proof of Theorem \ref{thm-1}}: Taking the limit as $b\rightarrow 1$ in \eqref{tl-1}, we obtain that for a positive integer $n>5$ with $n\equiv 1\pmod{4}$, modulo $\Phi_n(q)^2$,
	\begin{equation}
	\sum_{k=0}^{(3n+1)/4}[6k+1]\frac{(q;q^2)_k^3(q^{-1};q^4)_k}{(q^4;q^4)_k^3(q^{4};q^2)_k}q^{k^2+3k}
	\left(\sum_{i=1}^{k}\frac{q^{2i-1}}{[2i-1]^2}-\sum_{i=1}^{k}\frac{q^{4i}}{[4i]^2}\right) 
	\equiv 0.  \label{0-7}
\end{equation}	

In view of \cite[Page 9]{G-Z}, Guo showed that
\[
[6n-5]\frac{(q;q^2)_{n-1}^3(q^{-1};q^4)_{n-1}}{(q^4;q^2)_{n-1}(q^4;q^4)_{n-1}^3}q^{n^2+n-2}\equiv -\frac{[n]^3[3n](1+q)q^{(9n^2-5n)/4+1}}{[2n]} \pmod{\Phi_n(q)^4}.
\]

For $(3n+1)/4 < k \leq n-2$, note that the numerator $(q;q^2)_k^3(q^{-1};q^4)_k$ has the factor $(1-q^n)^3(1-q^{3n})$, while the denominator $(q^4;q^4)_k^3(q^4;q^2)_k$ and $[4i]^2$ for $1\leq i\leq k$ are coprime to $\Phi_n(q)$. However, the denominator $[2i-1]^2$ contains the factor $(1-q^n)^2$ when $i=(n+1)/2$. Hence, we have the following $q$-congruence:
\[
	\sum_{k=(3n+1)/4}^{n-2}[6k+1]\frac{(q;q^2)_k^3(q^{-1};q^4)_k}{(q^4;q^4)_k^3(q^{4};q^2)_k}q^{k^2+3k}
	\left(\sum_{i=1}^{k}\frac{q^{2i-1}}{[2i-1]^2}-\sum_{i=1}^{k}\frac{q^{4i}}{[4i]^2}\right) 
	\equiv 0 \pmod{\Phi_n(q)^2}.
\]
It then follows that
\begin{align}
\sum_{k=0}^{n-1}&[6k+1]\frac{(q;q^2)_k^3(q^{-1};q^4)_k}{(q^4;q^4)_k^3(q^{4};q^2)_k}q^{k^2+3k}
\left(\sum_{i=1}^{k}\frac{q^{2i-1}}{[2i-1]^2}-\sum_{i=1}^{k}\frac{q^{4i}}{[4i]^2}\right) \notag \\
&\equiv -[6n-5]\frac{(q;q^2)_{n-1}^3(q^{-1};q^4)_{n-1}}{(q^4;q^2)_{n-1}(q^4;q^4)_{n-1}^3}q^{n^2+n-2}\sum_{i=1}^{n-1}\frac{q^{2i-1}}{[2i-1]^2}  \notag\\
&\equiv -\frac{3q(1+q)[n]}{2} \pmod{\Phi_n(q)^2}, \label{0-8}
\end{align}
where we used \eqref{0-2-1} in the last step. Meanwhile, it can be readily verified that   \eqref{0-8} holds for $n=5$.

Thus, \eqref{th-1} follows from  \eqref{0-7} and \eqref{0-8}.

\noindent\textbf{Proof of Theorem \ref{thm-2}}: Taking the limit as $b\rightarrow 1$ in \eqref{tl-2}, we immediately arrive at \eqref{th-2}.

\section*{Declarations}

\begin{flushleft}
	\textbf{Conflicts of Interest:}	The author declares that he has no conflict of interest.\\[13pt]
	
	\textbf{Data Availability Statement:} Not applicable.
\end{flushleft}

\end{document}